\documentclass[12pt]{article}
\usepackage{amsfonts}

\usepackage{graphicx}
\usepackage{amsmath}

\newtheorem{theorem}{Theorem}

\newtheorem{corollary}[theorem]{Corollary}

\newtheorem{lemma}[theorem]{Lemma}

\newtheorem{proposition}[theorem]{Proposition}
\newtheorem{remark}[theorem]{Remark}

\newenvironment{proof}[1][Proof]{\textbf{#1.} }{\ \rule{0.5em}{0.5em}}

\begin{document}

\title{{\LARGE From the Sinai's walk to the Brox diffusion using bilinear forms.
\footnote{This research was partially supported by the CONACYT.}}}
\author{  
Carlos G. Pacheco
\thanks{Departamento de Matematicas, CINVESTAV-IPN, A. Postal 14-740, Mexico D.F. 07000, MEXICO. Email: cpacheco@math.cinvestav.mx}}

%\date{}
\maketitle

\begin{abstract}
Using the generators, we establish a connection between the Sinai's random walk and the so-called Brox process. 
We first find the Dirichlet form of the Brox diffusion, and then prove that it is the limit of the Dirichlet form of the Sinai's random walk.
This also gives a natural way to connect between the Brox diffusion and the Brownian motion.
\end{abstract}

\\


%\tableofcontents

\section{Introduction}

Sinai \cite{Sinai} studied the limiting behaviour of a random walk in random environment, and four years later Brox \cite{Brox} asked the same question for a diffusion with random coefficients. 
In both cases the same long time behaviour was found. 
In Seignourel \cite{Seignourel} it was shown the convergence in distribution of the Sinai's walk to the Brox diffusion.
Here we propose a modification of the scaling and a different method to connect both processes.
In particular, we what we do is to see that the generator of the Brox diffusion is the limit of the generator of the Sinai's walk.
This provides an straight forward way to establish the connection between these processes.
%Here, we consider the same question with a different perspective, which may bring a different understanding of the phenomenon. 
Moreover, this way to proceed helps to see the Brownian motion and the Brox process as models arising from the same microscopic system but with different local conditions.

An account regarding the connection of both models can be found in \cite{Shi}.
Other places where one finds invariance principles in the context of random media are \cite{Kawazu4}, where the random environment has a large drift; see \cite{Gu} to find work involving the so-called Brownian motion in random scenery; \cite{Mathieu2}, regarding random walks with random conductances in dimensions $2$ or more; and \cite{Mytnik}, regarding branching particle systems.

%More recently, based on \cite{Edelman}, it is proved in \cite{Ramirez} rigorously how a certain random matrix model is associated to a similar stochastic generator of the Brox's diffusion. It is then natural to think that the random matrix could be 

Let us briefly introduce an example of Sinai's walks (abbreviated in general RWRE), taken from R\'{e}v\'{e}sz \cite{Revesz}. 
Let $E:=\{p(z),\ z\in\mathbb{Z}\}$ be a family of i.i.d. random variables (called the environment) such that 
$P(p(z)=3/4)=P(p(z)=1/4)=1/2$. 
Then, the RWRE $R:=\{R_{k},\ k=1,2,\ldots\}$ has the following dynamics,
\begin{equation*}
P(R_{k+1}=y|R_{k}=z, E)=\left\{
\begin{array}{cr}
p(z) & \text{if }y=z+1\\
1-p(z) & \text{if }y=z-1.
\end{array}
\right.
\end{equation*} 
We may write $p(z)=1/2+q(z)$, with the following Bernoulli random variable 
\begin{equation*}
q(z):=\left\{
\begin{array}{cr}
1/4 & 1/2\\
-1/4 & 1/2.
\end{array}
\right.
\end{equation*} 
%$q(x)$ is such that $P(q(x)=1/4)=(q(x)=-1/4)=1/2$. 

Another way to present $R_{n}$ is the following:
\begin{equation*}
R_{n}=\sum_{i=1}^{n}\xi_{i}
\end{equation*}
where $\xi_{1},\xi_{2},\ldots$ are a sequence of random variables specified by
\begin{equation*}
\xi_{i}:=
\left\{
\begin{array}{cl}
1 & p\left(R_{i-1}\right) \\
-1 & 1-p\left(R_{i-1}\right).
\end{array}
\right.
\end{equation*}
%and for any $k\in \mathbb{Z}$, $p(k):=1/2+q(k)$. 
%Here, $\{q(k),\ k\in \mathbb{Z}\}$ is a sequence of independent Bernoulli random variables such that $P(q(k)=1)= P(q(k)=-1)=1/2.$
%We present the model using previous recursive expressions, because it is useful in the proofs, but it is also relevant to see that mathematically the classical random walk can be seen as a particular case when there is no random media.
This form goes more in hand with the way it is generally presented the classical random walk, and it turns out to help in our proofs.

When there is no random environment in the random walk we know that, after some rescaling, the central limit theorem gives rise to the normal random variable. 
On the other hand, Sinai \cite{Sinai} proved a limit theorem for the random walk where the environment is not being modified by any type of scaling. 
In the current paper we concern with the situation where the environment also suffers some rescaling.
Indeed, we consider random variables of the form
\begin{equation*}
\xi_{i}^{(n)}:=
\left\{
\begin{array}{cl}
1 & p_{n}\left(S_{i-1}\right) \\
-1 & 1-p_{n}\left(S_{i-1}\right),
\end{array}
\right.
\end{equation*}
where 
\begin{equation*}
p_{n}(x):=1/2+q(x)/n^{1/4}. 
\end{equation*}
Now set
\begin{equation*}
S_{0}:=0\text{ and }S_{k}^{(n)}:=\sum_{i=1}^{k}\xi_{i}^{(n)}. 
\end{equation*}
Then, in particular, we study limit behaviour, as $n\to\infty$, of 
$$S_{n}^{(n)}/\sqrt{n}.$$
The limit is not Gaussian as perhaps one might try to conjecture.
One can put this in contrast with the random variables
\begin{equation*}
R_{n}/(\log(n))^{2},
\end{equation*}
which converges in distribution to non-trivial random variable, see \cite{Sinai}.

Now, let us present the Brox diffusion, which strictly speaking is not even Markov, but one defines it first by conditioning to the environment, which indeed gives a diffusion. 
Define a continuous time stochastic process $\{X_{t},\ t\geq 0\}$ with continuous trajectories proposed through the expression
\begin{equation*}
dX_{t}=dB_{t}-\frac{1}{2}W^{\prime}(X_{t})dt,\ X_{0}=0,
\end{equation*}
where $B$ and $W$ are independent Brownian motions. This expression is meat to be stochastic differential equation with a random coefficient $W^{\prime}$. One way to see that process $X$ exists is by noticing that its generator would take the form
\begin{equation*}
\frac{1}{2}e^{W(x)}\frac{d}{dx}\left( e^{-W(x)} \frac{d}{dx}\right).
\end{equation*} 
Once one defines the conditioned process $X$ given an environment $W$, using the law of total probability, one defines what really the process $X$ is.

The aim of this paper is to prove that the Dirichlet form associated to the Sinai's walk converges to the corresponding one of the Brox diffusion.
To do that we need first to find the Dirichlet form of the Brox diffusion which is done in the following 2 sections. 
In Section 4, we present the sequence of Sinai's random walks that would help to approximate the Brox diffusion.
Then, in Section 5, we show the desired convergence.
There is an appendix at the end with a result we use at some point.

%\pagebreak
%\section{Preliminaries}

\section{The Brox diffusion}

Informally speaking, we call the stochastic process $X:=\{X_{t},\ t\geq 0\}$ the Brox diffusion, if it is solution of the equation
\begin{equation*}
dX_{t}=dB_{t}-\frac{1}{2}W^{\prime}(X_{t})dt,\ X_{0}=0.
\end{equation*}
Here $B:=\{B_{t},t\geq 0\}$ is a standard BM and $W:=\{W(x),x\in \mathbb{R}\}$ is a two-sided BM, they both independent from each other. 
This expression can be interpreted as a stochastic differential equation with a random coefficient $W^{\prime}$.

Based on standard theory of diffusions, in Brox \cite{Brox} it is properly defined $X$.
This is done arguing that
\begin{equation}\label{EqL}
L:=\frac{1}{2}e^{W(x)}\frac{d}{dx}\left( e^{-W(x)} \frac{d}{dx}\right)
\end{equation} 
is the infinitesimal generator associated to the equation displayed above. 
Thus, leaving fixed a trajectory of $W$, one is able to see that process $X$ is Markov process with generator denoted by $L$.
Moreover, this way of thinking corresponds to considering the scale function 
\begin{equation}\label{EqA}
A(x):=\int_{0}^{x}e^{W(y)}dy,\ x\in \mathbb{R}, 
\end{equation}
and the speed measure 
\begin{equation}\label{Eqm}
m(C):=\int_{C}2e^{-W(x)}dx,\ \text{ for }\text{Borel sets }C\subset\mathbb{R}.
\end{equation}
Then, using the scale function and a time-change transformation, Brox \cite{Brox} proposed the following explicit construction of $X$:
\begin{equation}\label{EqFX}
X_{t}=A^{-1}(B(T^{-1}(t, B))), \ t\geq 0,
\end{equation}
where 
\begin{equation}\label{FormulasT}
T(u, B):=\int_{0}^{u}e^{-2W(A^{-1}(B(s)))}ds,\ x\in\mathbb{R},\ u\geq 0.
\end{equation}
We can see that there are two sources of randomness, one coming from $B$ and the other from $W$.
We will say that $B$ is the \textit{intrinsic randomness} of $X$, and that $W$ is the \textit{external source of randomness}, i.e. the \textit{environment}.
We may leave fixed either $B$ or $W$ and study the random dynamics of the process $X$.
%We can see that for each fixed trajectory $W$ we have a well defined stochastic process whose randomness comes from $B$.
For instance, let $F$ the function in (\ref{EqFX}) that defines $X$ given $W$, i.e. 
\begin{equation}\label{EqF}
X=F(W).
\end{equation}
We could one step further and think of $F$ as a function of $B$ as well. 

%We denote by $P_{W}$ the quenched probability measure, which is determined after fixing a trayectory of $W$, i.e. it is the conditioned probability. And $P$ stands for the annealed probability measure.
If we leave fixed a trajectory $W$, one could write $X^{(W)}:=F(W)$ to emphasize that the process $X$ is conditioned to the environment $W$.
The corresponding probability measure of $X^{(W)}$ over $C[0,\infty)$ is denoted $P_{W}$ and it is called the \textit{quenched measure}. 
Whereas the probability measure on $C[0,\infty)$ coming from $X$ without leaving fixed $W$ is called the \textit{annealed probability}, and it is denoted by $P$. 
In other words, if $\mu$ is the probability measure over $C(\mathbb{R})$ associated to the two-sided Brownian motion $W:\Omega\to C(\mathbb{R})$, then
\begin{equation*}
P(C)= \int_{\Omega}P_{W(\omega)}(C)\mu(d\omega),
\end{equation*}
for any measurable set $C$ in $C[0,\infty)$.

To simplify notation, instead of writing $X^{(W)}$, with no risk of confusion we only write $X$, because we will in general be working with the stochastic process $X$ after leaving fixed an environment $W$, i.e. the quenched case.

%The fact that $X$ is a diffusion after leaving fix an environment $W$ comes from the following theorem that one can extract from the literature.
%\begin{theorem}
%Let $\mu$ be a $\sigma$-finite measure on $\mathbb{R}$. 
%Then, the process $Y_{t}:=B_{\gamma_{t}}$ is a regular diffusion, where $\gamma_{t}$ is the inverse function of
%\begin{equation}\label{DefT2}
%t\mapsto T_{t}:=\int_{-\infty}^{\infty}L_{t}(x)\mu(dx),
%\end{equation}
%with $L_{t}(x)$ being the local time of $B$ at point $x$ up to time $t$.
%\end{theorem}
%\begin{proof}
%See point (ii) of the Remark in \cite[p.277]{Rogers2}, together with III.21 in \cite[p.277]{Rogers1}.
%\end{proof}

%\begin{corollary}
%For each trajectory $W$, process $X$ is a diffusion.
%\end{corollary}
%\begin{proof} 
%Since function $A$ in (\ref{EqA}) is continous and strictly increasing, it is enough to see that $B(T^{-1}(t))$ in (\ref{EqFX}) is a diffusion. 
%But this is the case because $T(u)$ plays the role of $T_{u}$ in (\ref{DefT2}). 
%Indeed, using properties of the local time of $B$ (see (2.90) in \cite[p.33]{Marcus}) we can see that the measure $\mu$ is specified by 
%\begin{equation*}\label{EqSpeedM}
%\mu(C):=\int_{C}e^{-2W(A^{-1}(x))}dx, \text{ for Borel sets }C\subset\mathbb{R}. 
%\end{equation*}
%which is the speed measure of .
%\end{proof}

Now, since $X$ is a Markov process for each fixed environment $W$, there is a semigroup $\{H_{t},\ t\geq 0\}$ defined as
\begin{equation}\label{EqSemi}
H_{t}f(x):=E[f(X_{t})|X_{0}=x],
\end{equation}
with $f\in C_{0}$, the space of real-valued continuous functions vanishing when $|x|\to \infty$. 
Symbollically, the generator of such semigroup is $L$ in (\ref{EqL}), and let $D$ be the domain of $L$.
%, which is known to be dense in $C_{0}$.
Observe that the domain $D$ depends on the environment $W$.

Before we continue, let us mention that in the rest of the paper we denote by $C_{0}$ the space $\mathbb{R}\to\mathbb{R}$ of continuous functions that vanish as $|x|\to \infty$, and by $C_{0}^{k}$ the subspace of $k$-times differentiable functions.

\begin{remark}
Notice that at first glance $Lf$ seems to be undefined, because $dW/dx$ does not exist, but since $X^{(W)}$ is indeed a diffusion, there is indeed a generator $L$ with domain $D$, and such domain is known to be dense in $C_{0}$ 
(see e.g. \cite[Lemma 3, p. 23]{Mandl}). 
Therefore, $Lf$ is well defined for any $f\in D$.
\end{remark}

The following result will be useful for our purposes.
\begin{proposition}\label{PropDom}
For any environment $W$, the domain $D$ is contained in the space of differentiable functions $C^{1}(\mathbb{R})$.
\end{proposition}
\begin{proof}
According to Lema 2 in \cite[p.22]{Mandl}, if $g(x):=[Lf](x)$ for $f\in D$ then
\begin{equation*}
f(x)=\int_{0}^{x}\int_{0}^{y}g(s)dm(s)dA(y)+f(0)+ A(x)\frac{df}{dA}(0),
\end{equation*}
where 
\begin{equation}
\frac{df}{dA}(0):=\lim_{y\to 0}\frac{f(y)-f(0)}{A(y)-A(0)}=\lim_{y\to 0}\frac{f(y)-f(0)}{\int_{0}^{y}e^{W(s)}ds}\frac{y}{y}
= \frac{f^{\prime}(0)}{e^{W(0)}}= f^{\prime}(0).
\end{equation}
Thus, pluging the scale funtion (\ref{EqA}) and the speed measure (\ref{Eqm}), we have that
\begin{equation*}\label{Eqfdiff}
f(x)=2\int_{0}^{x}\int_{0}^{y}g(s)2e^{-W(s)}e^{W(y)}dsdy+f(0)+ \frac{df}{dA}(0) \int_{0}^{x}e^{W(y)}dy.
\end{equation*}
This tells us that $f$ is indeed differentiable at every $x\in \mathbb{R}$.
%\begin{equation}
%\frac{df}{dA}(0):=\lim_{y\to 0}\frac{f(y)-f(0)}{s(y)-s(0)}=\lim_{y\to 0}\frac{f(y)-f(0)}{\int_{0}^{y}e^{W(s)}ds}.
%\end{equation}
%Since $\lim_{y\to 0}\frac{y-0}{\int_{0}^{y}e^{W(s)}ds}$ exists, from previous display we know that $\lim_{y\to 0}\frac{f(y)-f(0)}{y-0}$ exists. 
%We can then conclude from (\ref{Eqfdiff}) that any $f\in D$ is differentiable.
\end{proof}

%\begin{corollary}
%The martingale problem for $L$ posed on $D\bigcap C_{0}^{2}$ has a unique solution given precisely by the Brox diffusion.
%\end{corollary}
%\begin{proof}
%Due to previous proposition, $D\subset C_{0}^{1}$. 
%Since $C_{0}^{2}$ is dense in $C_{0}^{1}$.
%\end{proof}

%In order to deal with the bilinear form associated to the generator, we consider $L$ as acting on subspaces of the Hilbert space $L_{2}(\mathbb{R})$. 
%It is known that $L$ is well defined over $C_{0}^{\infty}(\mathbb{R})\subset D$ which is dense on $L_{2}(\mathbb{R})$.

%Let us now mention regarding the domain $D$ of the generator. 
%As it is pointed out in \cite{Ito}, the degree of the generator is less or equal to $2$.
%This in turn implies that $D$ ...

%\pagebreak
\section{Approximations and the Dirichlet form}\label{SecAprox}

%In this section we approximate the environment with processes where it is easier to see the structure of the associated martingale problem.
Let $\{W_{n}\}$ be a sequence of stochastic processes that converge weakly to $W$.
Then, by considering the map $F$ in (\ref{EqF}), one can study the processes $X^{(n)}:=F(W_{n})$.
If the trajectories of $W_{n}$ are differentiable functions, $X^{(n)}$ can be analyzed using standard tools.
In particular, we can approximate the Dirichlet form of $X$ by considering the Dirichlet form of $X^{(n)}$.

%In particular, the associated martingale problem of $X^{(n)}$ will be rewritten to yield the structure we need using local times. 

Let us start. 
Let $W_{n}, n=1,2,\ldots$ be a sequence of processes with piecewise differentiable paths and such that $W_{n}$ converges weakly to $W$ as $n\to\infty$.
As it was done for (\ref{EqA})-(\ref{FormulasT}), define
\begin{equation}
X_{t}^{(n)}(B):=A_{n}^{-1}(B(T_{n}^{-1}(t, B))),
\end{equation}
where 
\begin{equation*}
A_{n}(x):=\int_{0}^{x}e^{W_{n}(y)}dy, \ x\in\mathbb{R}
\end{equation*}
and 
\begin{equation*}
T_{n}(u, B):=\int_{0}^{u}e^{-2W_{n}(A_{n}^{-1}(B(s)))}ds,\ u\geq 0.
\end{equation*}
For each $n=1,2, \ldots$ we denote by $H_{t}^{(n)}$ the associated semigroup of $X^{(n)}$ after leaving frozen the environment $W_{n}$.
After finishing our paper, we came across with the manuscript \cite{Hu}, where approximation of this type are also used to study the Brox diffusion.

The generator of $X^{(n)}$ is given by
\begin{equation}
[L^{(n)}f](x):=\frac{1}{2}e^{W_{n}(x)}\frac{d}{dx}\left( e^{-W_{n}(x)}\frac{d}{dx} \right)
=\frac{1}{2}\frac{d^{2}f}{dx^{2}}-\frac{W_{n}^{\prime}(x)}{2}\frac{df}{dx}.
\end{equation}
Then we can describe $L^{(n)}$ using the bilinear form
\begin{equation*}
\langle L^{(n)}f,g \rangle =
 -\frac{1}{2}\int_{-\infty}^{\infty} f^{\prime}g^{\prime} - \frac{1}{2}\int_{-\infty}^{\infty} f^{\prime}gdW_{n},
\end{equation*}
for any $f,g\in C_{0}^{1}$, the space of real-valued differentiable functions that vanishes as $|x|\to\infty$.
%We want to check that previous expression carries on in the limit when $n\to\infty$.
\begin{center}
Our aim is to analyze $\langle L^{(n)}f,g \rangle$ as $n\to\infty$.
\end{center}
%prove that the sequence of martingales $M^{(n)}(f)$ converge in certain sence to the martingale $M (f)$ for each $f$ in $C_{0}^{2}\bigcap D$.
%This will lead us to see what the martingale problem is for the Brox diffusion.  
%To simplify notation we only write $M$ and $M^{(n)}$ for any $f\in C_{0}^{\infty}(\mathbb{R})$.

%We now want to check that this approximation really helps to infer the bilinear form associated to the Brox's diffusion. 
%\pagebreak

Recall that a trajectory of $W$ is being fixed. 
In addition, one can go one step further and consider $X$ to be a function of each trajectory $B\in C[0,\infty)$. 
Indeed, consider the map 
\begin{equation*}
(t,B)\mapsto X_{t}(B),
\end{equation*}
and the same for $X^{(n)}$. 
As usual, we take $C[0,\infty)$ with the topology of uniform convergence in compact sets. 
The first thing we need to prove is that $X^{(n)}$ converges uniformly to $X$, in compact sets of the arguments $(t, B)$. 

\begin{lemma}\label{lemaXnX}
We have that
\begin{equation*}
X_{t}^{(n)}(B)\to X_{t}(B),\text{ as }n\to\infty,
\end{equation*} 
uniformly in compact sets of $(t,B)\in[0,\infty)\times C[0,\infty)$.
\end{lemma}
\begin{proof}
To prove convergence uniformly in compact sets, it is a topological fact (see e.g. Chapter XII, section 7, p.267 of \cite{Dugundji}) that it is enough to prove that
\begin{equation}\label{EqCXn}
X_{t_{n}}^{(n)}(B_{n})\to X_{t}(B),\text{ as }n\to\infty,
\end{equation} 
whenever $t_{n}\to t$ and $B_{n}\to B$.

First of all, since $W_{n}\to W$ uniformly in compact sets, then $A_{n}\to A$ pointwise, and by Theorem \ref{ThmUCC} in the Appendix, the convergence is also uniform in compact sets. 
Therefore, $A_{n}^{-1}\to A_{n}^{-1}$ uniformly in compact sets, see Lemma 6 in \cite{Seignourel}. 
From this, we can see that the composition 
$$W_{n}\circ A_{n}^{-1}\circ B_{n}\text{ also converges to }W\circ A^{-1}\circ B$$
uniformly in compact sets of the domain $[0,\infty)$.
This implies that
\begin{equation}\label{LimTn}
T_{n}(t_{n}, B_{n})=\int_{0}^{t_{n}}e^{-2 W_{n}(A_{n}^{-1}(B_{n}(s)))}ds\to 
T (t, B)=\int_{0}^{t}e^{-2 W(A^{-1}(B(s)))}ds, \ n\to\infty.
\end{equation}
With all these we can see that 
\begin{equation*}
A_{n}^{-1}\circ B_{n}\circ T_{n}^{-1} (t_{n}, B_{n})\to A^{-1}\circ B \circ T^{-1}(t,B), \ n\to\infty,
\end{equation*}
which is (\ref{EqCXn})
\end{proof}

\begin{corollary}\label{CoroFcont}
For almost every $B$, the map 
$$F:C(\mathbb{R})\to C[0,\infty)$$ 
described in (\ref{EqF}) is continuous.
\end{corollary}
\begin{proof}
First, fix $B$. We need to show that $F(W_{n})\to F(W)$ whenever $W_{n}\to W$ as $n\to\infty$. 
Let $F_{t}(W_{n}):=X_{t}^{(n)}$ and $F_{t}(W):=X_{t}$.
As mentioned in the proof of Lemma \ref{lemaXnX}, it is enough to prove that $F_{t_{n}}(W_{n})\to F_{t}(W)$ whenever $t_{n}\to t$, statement which is contained in previous lemma.
\end{proof}

\begin{lemma}\label{LemaSGc}
For each $t\geq 0$ and $f\in C_{0}$,
$$\|H_{t}^{(n)}f- H_{t}f\|\to 0,\ n\to \infty.$$ 
\end{lemma}
\begin{proof}
From Lemma \ref{lemaXnX}, we have that $X_{t}^{(n)}(B)\to X_{t}(B)$ for each path $B$, we are then saying that $X_{t}^{(n)}\to X_{t}$ for almost every path $B$.
Then, 
$$X_{t}^{(n)}+x_{n}\to X_{t}+x$$ 
almost surely whenever $x_{n}\to x$. 
Therefore, because any $f$ is bounded,
\begin{equation*}
E\left[ f(X_{t}^{(n)}+x_{n}) \right]\to E\left[ f(X_{t}+x) \right],\ n\to\infty.
\end{equation*} 
%So, if $H_{t}^{(n)}$ and $H_{t}$ represents the semigroup associated to $X^{(n)}$ and $X$, respectively, 
So, we are actually saying that $H_{t}^{(n)}f$ converges to $H_{t}f$ uniformly in compact sets of $\mathbb{R}$. 
%Now we want to see that $T_{t}^{(n)}f$ converges to $T_{t}f$ uniformly in the whole real line.
However, since $f$ vanishes at the infinity, given $\epsilon>0$ we can give a compact set $K\subset \mathbb{R}$ such that
\begin{equation}
\left| E[f(X_{t}^{(n)}+x)]-E[f(X_{t}+x)]  \right|\leq \left| E[f(X_{t}^{(n)}+x)]-E[f(X_{t}+x)]  \right|I_{K}(x)+\epsilon,
\end{equation}
where $I_{K}(x)=0$ if $x\notin K$ and $I_{K}(x)=1$ if $x\in K$. 
Therefore, we have that $H_{t}^{(n)}f$ converges to $H_{t}f$ uniformly in the whole real line, i.e. with the supremum norm.
\end{proof}

We can now prove that $Lf_{n}\to Lf$ in the supremum norm for those functions $f$ where $Lf$ is well defined, i.e. if $f\in D$.
\begin{corollary}\label{CoroLnL}
It holds that 
$$\|L_{n}f-Lf\|\to 0$$ 
for any $f\in  D$.
\end{corollary}
\begin{proof}
It follows from Lemma \ref{LemaSGc} and Theorem 6.1 in \cite[p.28]{Ethier} that 
$$\|L_{n}f-Lf\|\to 0\text{ for any } f\in D,$$
in particular for any $f\in D$.
\end{proof}

We now present the so-called Dirichlet form associated to the Markov process.
\begin{theorem}
We have that the generator $L$ of the Brox diffusion satisfies
\begin{equation*}
\langle Lf,g \rangle =
 -\frac{1}{2}\int_{-\infty}^{\infty} f^{\prime}g^{\prime} - \frac{1}{2}\int_{-\infty}^{\infty} f^{\prime}gdW,
\end{equation*}
for functions $f,g\in C_{0}^{1}$. 
\end{theorem}
Notice that since $f$ and $g$ are deterministic functions, 
the last integral of previous display can be considered an It\^{o}'s stochastic integral.

\begin{proof}
From Corollary \ref{CoroLnL} 
$$(\forall f\in C_{0}^{2}\bigcap D)\ L_{n}f\to Lf$$
with the supremum norm.
%Due to Theorem \ref{ThmCore}, $C_{0}^{2}\bigcap D$ is dense in $C_{0}$, 
Thus, 
$$(\forall f\in C_{0}^{2}\bigcap D)(\forall g\in C_{0}) \langle L_{n}f, g\rangle\to\langle Lf, g\rangle.$$

On the other hand, from the theory of stochastic integrals (by taking a subsequence from the $L_{2}$-convergence) we can take the convergence
\begin{equation*}
\int f^{\prime} g dW_{n}\to \int f^{\prime} g d W
\end{equation*}
to be almost surely, i.e. for almost every $W$.
Thus, we can say that
\begin{equation*}
\langle L_{n}f, g\rangle \to 
-\frac{1}{2}\int_{-\infty}^{\infty} f^{\prime}g^{\prime} - \frac{1}{2}\int_{-\infty}^{\infty} f^{\prime}gdW, \
n\to\infty,
\end{equation*}
which in fact holds for functions $f,g\in C_{0}^{1}$. 
However, according to Proposition \ref{PropDom}, the domain of $L$ is contained in $C_{0}^{1}$. 
Therefore, the bilinear form is the associated Dirichlet form of the generator $L$.
\end{proof}

In previous result, one realizes that, when describing $L$ through the bilinear form, the domain is considered to be $C_{0}^{1}$, which 
is a bigger set of functions and independent of the trajectory $W$. 
This kind of phenomenon is part of the theory of Dirichlet forms.  
The interested reader on Dirichlet forms and Markov process can consult \cite{Fukushima}.

\section{The Sinai's walk}
In this section we want to introduce the sequence of Sinai's random walks that approximate the Brox diffusion. 
We will see that apart from rescaling time and space of the corresponding random walk, we also need to rescale the the transition probabilities, i.e. the random environment.

Consider a partition of $\mathbb{R}$ into equally spaced intervals of size $\Delta_{n}>0$, and denote by $\mathbb{Z}_{n}\subset\mathbb{R}$ the lattice given by the set of extreme points of the intervals. 
Do now the same with the open interval $[0,\infty)$ by considering the discrete time $T_{n}\subset [0,\infty)$ of equidistant point of size $h_{n}>0$. 
In this paper we will focus in a rescaling such that 
$$\Delta_{n}:=1/\sqrt{n}\text{ and }h_{n}:=1/n.$$ 
We must warn the reader not to be confused by expressions such as $\Delta_{n}x$, which does not mean an increment on $x$  but it is just the multiplication $\Delta_{n}\times x$.
%but we think it helps to carry with the symbols $\Delta$ and $h$ along the calculations.

Now we present the sequence of approximating random walks, where the notation is such that the model fits into the description of the introduction.

%The notation will seem a bit long but we think it helps to have a better general picture. 
For each $n=1,2,\ldots$ consider the following continuous time stochastic process
\begin{equation*}
S_{t}^{(n)}:=\sum_{i=0}^{\left[t/h_{n}\right]}\xi_{i}^{(n)},\ t\geq 0,
\end{equation*}
where $[x]$ is the interger part of $x$ and 
\begin{equation*}
\xi_{i}^{(n)}:=
\left\{
\begin{array}{cc}
+1 & p_{n}\left(S_{i-1}^{(n)}\right)\\
-1 & 1-p_{n}\left(S_{i-1}^{(n)}\right),
\end{array}
\right.
\end{equation*}
with $p_{n}(k):=1/2+\sqrt{\Delta_{n}}q (k)$, with
\begin{equation*}
q(z):=\left\{
\begin{array}{cr}
1/4 & 1/2\\
-1/4 & 1/2.
\end{array}
\right.
\end{equation*} 

Notice that $E\left[ \xi_{i}^{(n)} \xi_{j}^{(n)} \right]=0$ as long as $S_{i-1}^{(n)}\neq S_{j}^{(n)}$, and 
$E\left[ \xi_{i}^{(n)} \xi_{j}^{(n)} \right]=n^{-1/2}$ otherwise. 
This information will be useful to know later.

%and the increments are given by 
Now, let 
\begin{equation*}
X_{t}^{(n)}:=\frac{1}{\sqrt{n}}S_{t}^{(n)}.
\end{equation*}
%$\xi_{i}^{(n)}:=\Delta_{n} \zeta_{i}^{(n)},\ i=1,2,\ldots$, 
%and define $X_{t}^{(n)}:=\frac{1}{\sqrt{n}}S_{t}^{(n)}$.

%Moreover, the jumping probabilities are given by $p_{n}(x_{i}):=1/2+\sqrt{\Delta_{n}}q (x_{i})$, where $x_{i}$ and $q$ are defined below. 
%Let us first comment on these quantities.
%As it was described in the introduction, the jumping probabilities $p_{n}$ depend on the position of the particle. 
%One may think that the particle has to see where it is to see with which coin is going to make its next movement.
%For example, suppose that in the $i$-movement, the particle is at site $x\in \mathbb{Z}_{n}$. 
%Notice that $x/\Delta_{n}$ is an integer, then for the $i+1$-increment the jumping probability is going to be determined by certain random variable $q(x/\Delta_{n})$, defined below. 
%The following specifications clarify this description,
%\begin{equation*}
%x_{i}:=\frac{X_{ih_{n}}^{(n)}}{\Delta_{n}}\in \mathbb{Z},
%\end{equation*}
%and
%\begin{equation*}
%q (x):=
%\left\{
%\begin{array}{cc}
%+1 & 1/2\\
%-1 & 1/2.
%\end{array}
%\right.
%\end{equation*}
%for any $x\in \mathbb{Z}$.

For any $n=1,2,\ldots$, one can see that the random environment $E_{n}:=\{p_{n}(z),\ z\in \mathbb{Z}\}$ is known once the sequence $E:=\{q(z), \ z\in\mathbb{Z}\}$ is specified.
That is, given $E$ we know all the values in $E_{n}$.
Moreover, from the Donsker invariance principle we know that
\begin{equation}\label{EqEtoW}
\left\{ \sqrt{\Delta_{n}} \sum_{j=0}^{k/\Delta_{n}} q(j) ,\ k\in \mathbb{Z} \right\}\overset{d}{\to} 
\left\{ \frac{1}{4}W(t), \ t\in \mathbb{R} \right\},\ n\to\infty,
\end{equation}
where the factor $1/4$ was factorized from $q$.
This tells us that we can associate to each sample of $E$ a trajectory of $W$, or the other way round.
This fact will be used in later in Theorem \ref{LemaLnL}.

%This simple fact tells us that ...

Given the environment $E$, when we consider the Sinai's walk $X_{t}^{(n)}$ only at the jumps, i.e. for $t\in T_{n}$.
This gives a discrete time Markov chain.
In this case, the generator $L^{(n)}$ of $\{X^{(n)}_{t},t\in T_{n} \}$ acting on a function $f:\mathbb{Z}_{n}\to \mathbb{R}$ is given by
\begin{equation}\label{Ln}
[L^{(n)}f](x)=\frac{1}{h_{n}}\left( f(x+\Delta_{n})p_{n}(x/\Delta_{n} )+f(x-\Delta_{n})(1-p_{n}(x/\Delta_{n}))-f(x) \right),\ x\in \mathbb{Z}_{n}.
\end{equation} 

%\pagebreak
\section{Convergence of the Dirichlet form}\label{SecDist}

Due to previous sections, we are now in position to prove our main result, which is the convergence of the Dirichlet forms.

Remember that we are dealing with RWREs that takes values in $\mathbb{Z}_{n}\subset \mathbb{R}$. 
From (\ref{Ln}), we can think that $ L^{(n)}f $ is a step valued function from $\mathbb{R}$ to $\mathbb{R}$.
This corresponds to consider $L^{(n)}$ as the composition of two operators, one is a projection to a vector and the other is multiplying such projection with a matrix; this way of thinking was imported from Pacheco \cite{Pacheco}.

Then, for every pair $f,g\in C_{0}^{1}$,
\begin{eqnarray*}
\langle L^{(n)}f,g \rangle&=&
\int_{-\infty}^{\infty}L^{(n)}f(y)g(y)dy\\ 
&=&\sum_{x\in \mathbb{Z}_{n}} \left\{\frac{1}{h_{n}}\left( f(x+\Delta_{n})p_{n}(x/\Delta_{n})
+f(x-\Delta_{n})(1-p_{n}(x/\Delta_{n}))-f(x) \right)\int_{x}^{x+\Delta_{n}}g(y)dy\right\}.
\end{eqnarray*}

Now, we can prove that when leaving fixed an environment $W$, $\langle L^{(n)}f,g \rangle  \to  \langle Lf,g \rangle$ as $n\to\infty$.

%Now, since $X^{}$it is embedded in the discrete time $\{0, h_{n}, 2h_{n}, 3h_{n},\ldots\}\subset [0,\infty)$.
%Since we can set the generator of processes $X^{(n)}$ and $X$, we may tray to compare them. 
%We have the following weak convergence (in the sense of Skorohod \cite{Skorohod}) of $L^{(n)}$ to $L$.

\begin{theorem}\label{LemaLnL} 
There is a subsequence $\{n_{k}\}_{k\geq 1}$ such that
\begin{equation*}
\langle L^{(n_{k})}f, g \rangle\to
\frac{1}{2}\int_{-\infty}^{\infty} f^{\prime\prime}g - \frac{1}{2}\int_{-\infty}^{\infty} f^{\prime}gdW,\ k\to\infty,
\end{equation*}
for almost every trajectory $W$ and for functions $f\in C_{0}^{2}$ and $g\in C_{0}$.
\end{theorem}
\begin{proof}
To simplify notation let us omit subscript $n$ in $\Delta_{n}$ and in $p_{n}$. Using the mean valued theorem for integrals (see e.g. Bartle\cite{Bartle}), there are values $y_{x}\in [x,x+\Delta)$ for each $x\in \mathbb{Z}_{n}$ such that
$$ \int_{x}^{x+\Delta}g = g(y_{x})\Delta.$$
Define $q_{n}(x):=\sqrt{\Delta}q(x/\Delta )$. Then we have
\begin{eqnarray*}
\langle L^{(n)}f,g \rangle &=& 
\sum_{x\in \mathbb{Z}_{n}}\left\{\frac{1}{\Delta^{2}}\left( f(x+\Delta)p(x/\Delta )+f(x-\Delta)(1-p(x/\Delta ))-f(x) \right)g(y_{x})\Delta\right\}\\
&=& \sum_{x\in \mathbb{Z}_{n}}\left\{\frac{1}{\Delta^{2}}\left( \left\{\frac{1}{2}+q_{n}(x)\right\}\{f(x+\Delta)-f(x-\Delta)\}+ f(x-\Delta) -f(x) \right)g(y_{x})\Delta\right\}\\
&=& \frac{1}{2}\sum_{x\in \mathbb{Z}_{n}}\frac{f(x+\Delta)-2f(x)+f(x-\Delta)}{\Delta^{2}}g(y_{x})\Delta \\
&+& 2\sum_{x\in \mathbb{Z}_{n}} \frac{f(x+\Delta)-f(x-\Delta)}{2\Delta}g(y_{x})q_{n}(x).
\end{eqnarray*}

Notice that the variance of the random variables $q_{n}(x)$ is of order $\Delta$ and they from an independent sequence.
Hence, using theory of stochastic integrals, and using the fact (\ref{EqEtoW}), one can check that as $n\to\infty$ 
%previous display converges in mean square to expression (\ref{Def0L}), with $W$ being a two-sided BM. 
for any pair $f,g\in C_{0}^{2}$
\begin{equation*}
\langle L^{(n)}f, g \rangle\overset{L_{2}}{\to}
\frac{1}{2}\int_{-\infty}^{\infty} f^{\prime\prime}g - \frac{1}{2}\int_{-\infty}^{\infty} f^{\prime}gdW,
\end{equation*}
where $\overset{L_{2}}{\to}$ stands for convergence in mean square of random variables.
In previous limit we have factorized the constant $1/4$ from the Bernoulli random variable $q_{n}(x)$.
Thus, we can suppose that there is a subsequence where we have convergence almost surely, that is for almost every fixed trajectory $W$.
\end{proof}

One can also check from the proof the following.
\begin{remark} 
In the proof of previous theorem, suppose that $q_{n}$ are random variables with $var(q_{n}(x))\leq c \Delta^{\gamma}$ for some constants $c>0$ and $\gamma\in[0,1)$. 
Then, in the limit the second term vanishes and one recovers the generator of the Brownian motion. 
In this way, we can think of the Brownian motion and the Brox diffusion coming from the same type of model with different local specifications.
\end{remark}

\begin{remark} By taking $q_{n}(x)$ such that
\begin{equation*}
q_{n}(x)=
\left\{
\begin{array}{cr}
(\sqrt{\Delta_{n}}+\kappa \Delta_{n})/4 & 1/2\\
(-\sqrt{\Delta_{n}}+\kappa \Delta_{n})/4 & 1/2,
\end{array}
\right.
\end{equation*} 
we can have the convergence when the environment $W(x)$ is of the form $\beta(x)+\kappa x$, where $\beta$ is another BM. 
This model is considered for instance in \cite{Kawazu3, Talet}.
\end{remark}
 
\section{Appendix}

The following result turns out to be very useful to us. 
It can be found as an exercise 127 in \cite[p. 81]{Polya}. 
%One can see it as an exercise in ... We thank Micheal K. Porter for pointing out this.
\begin{theorem}\label{ThmUCC}
Let $f_{n}:\mathbb{R}\to \mathbb{R}$ be a sequence of continuous functions that converge pointwise to another continuous function $f:\mathbb{R}\to \mathbb{R}$.
If in addition all these functions are monotone, then the convergence is uniform in compact sets.
\end{theorem}

To prove this result one could use the following idea.
We need to prove that $f_{n}(x_{n})\to f(x)$ whenever $x_{n}\to x$.
Since 
$$
|f_{n}(x_{n})-f(x)|\leq |f_{n}(x_{n})-f_{n}(x)|+|f_{n}(x)-f(x)|,
$$
we concentrate on proving that $|f_{n}(x_{n})-f_{n}(x)|\to 0$.

If the convergence were not uniformly, given $\epsilon >0$ we could take $\{n_{k}\}\subset \{n\}$ such that
\begin{equation}\label{IneqAppend}
|f_{n}(x_{n_{k}})-f_{n_{k}}(x)|>\epsilon,\ k=1,2,\ldots.
\end{equation}

Now, since each $f_{n_{k}}$ is continuous, there are open sets $V_{k}$ such that 
$$(\forall y\in V_{k})|f_{n_{k}}(y)-f_{n_{k}}(x)|<\epsilon.$$
Therefore, each $x_{n_{k}}$ cannot be inside $V_{k}$.

From the monotonicity assumption, one knows that the sets $V_{k}$ need to be convex intervals.
But the fact $x_{n_{k}}\to x$, as $k\to\infty$, would imply that the length of $V_{k}$ goes to $0$.
This would contradict (\ref{IneqAppend}).

\bigskip

\textbf{Acknowledgements.} 
%We thank R. Michael Porter for providing the reference \cite{Polya}.

%\pagebreak

%\textbf{Acknowledgements}. 
%\newpage

\end{document}